\documentclass[11pt]{article}
\usepackage{amsmath,amssymb,amsthm}
\usepackage[margin=2.5cm]{geometry}
\usepackage{enumitem}
\usepackage{hyperref}
\usepackage{booktabs}
\usepackage{tikz}
\usepackage{xcolor}
\newcommand{\rev}[1]{#1}
\newtheorem{theorem}{Theorem}[section]
\numberwithin{equation}{section}
\newtheorem{lemma}[theorem]{Lemma}
\newtheorem{proposition}[theorem]{Proposition}
\newtheorem{corollary}[theorem]{Corollary}

\theoremstyle{remark}
\newtheorem{remark}[theorem]{Remark}
\newtheorem{example}[theorem]{Example}

\title{The domination number of the $2$-token graph\\ of the complete bipartite graph}

\author{E. Acosta Troncoso, I. Caldera Caldera, and L. M. Rivera}
\date{}

\begin{document}
\maketitle

\begin{abstract}
\noindent
For a graph $G$, the $2$-token graph $F_2(G)$ has as vertices all $2$-subsets of $V(G)$, with two such subsets adjacent whenever their symmetric difference is an edge of $G$. In this paper, we determine the domination number of the $2$-token graph of the complete bipartite graph $K_{m,n}$, for all integers $2\le m\le n$. We show that it equals $\lceil m/2\rceil+n-1$, except when $m=3$, or when $m\ge5$ is odd and $n=m+1$; in these two exceptional situations it is one less. Our arguments rely on an explicit characterization of the dominating sets of the $2$-token graph of $K_{m,n}$.


\end{abstract}

\section{Introduction}

Let $G=(V,E)$ be a simple graph and let $1\le k\le |V|-1$ be an integer. The \emph{$k$-token graph} $F_k(G)$ is the graph whose vertex set is $\binom{V}{k}$, the collection of all $k$-subsets of $V$, where two $k$-subsets are adjacent whenever their symmetric difference is an edge of $G$. Equivalently, $F_k(G)$ models $k$ indistinguishable tokens placed on the vertices of $G$ (at most one per vertex), with adjacency given by sliding a single token along an edge to an unoccupied vertex~\cite{fabila2012}.

The graph $F_2(G)$, with the name of the \emph{double vertex graph} of $G$, was introduced by Alavi, Behzad, Erd\H{o}s and Lick~\cite{alavi1991} and has been studied extensively since then (see the survey~\cite{alavi2002}). However, it can be traced back to Johns' 1988 thesis under the name of $2$-subgraph graph~\cite{johns1988}. In 2002, Rudolph~\cite{rudolph2002} rediscovered $F_2(G)$ as the symmetric square of $G$, motivated by the graph isomorphism problem and by applications in quantum mechanics. The general $k$-token graph was introduced under its current name by Fabila-Monroy, Flores-Pe\~{n}aloza, Huemer, Hurtado, Urrutia and Wood~\cite{fabila2012}, who studied several of its combinatorial invariants (connectivity, diameter, clique and chromatic number, Hamiltonicity). Since then, token graphs have received considerable attention; see, e.g.,~\cite{carballosa2017,dalfo10, dalfo11,dealba2020, fabila6, lew}.

A set $D\subseteq V(G)$ is \emph{dominating} if every vertex of $G$ is in $D$ or has a neighbor in $D$; the \emph{domination number} $\gamma(G)$ is the minimum cardinality of such a set. Domination is one of the oldest and most extensively studied parameters in graph theory, and computing it is NP-hard in general~\cite{garey1979}.
 
The domination number of token graphs has become an active line of research only recently. Prabhu and Manjula~\cite{prabhu2019} obtained general bounds for $\gamma(F_2(G))$ in terms of structural parameters of the graph $G$. G\'omez Soto, Lea\~nos, R\'ios Castro, and Rivera~\cite{gomezsoto2018} determined the packing number of $F_2(P_n)$ for the path graph $P_n$, which, by the standard inequality $\gamma(G)\ge \rho(G)$, yields a lower bound for $\gamma(F_2(P_n))$. Recently, Acosta Troncoso, Lea\~nos, and Rivera~\cite{acosta2026} matched this lower bound asymptotically \rev{and computed} the exact values for $n\le 25$ via computer search. In a related direction, Fabila-Monroy and G\'omez-Galicia~\cite{fabila2025} determined $\gamma(F_2(G))$ exactly when $G$ is a complete graph $K_n$ or a star $S_n$ (as well as $\gamma(F_3(K_n))$).

To our knowledge, the domination number of the $2$-token graph of complete bipartite graphs $K_{m,n}$ had not been addressed before, although other parameters of the token graphs of $K_{m,n}$, such as the automorphism group~\cite{fabila6}, the independence and matching numbers, have been studied before~\cite{dealba2020}. In this paper, we determine $\gamma(F_2(K_{m,n}))$ exactly for every pair of integers $2\le m\le n$.

Our main result is the following:\footnote{A first version of this paper was submitted for publication before the preprint of Chinman, Cortijo, Pedroni, Seong and Stenger~\cite{CCPSS} appeared on arXiv on October 20, 2026; both works were done independently, and they obtain Theorem~\ref{thm:main} by different methods. The dominating sets of~\cite{CCPSS} are essentially those of Sections~\ref{sec:even}--\ref{sec:odd}, but their lower bounds are proved in a different way. Some general bounds for $\gamma(F_k(K_{m,n}))$ with $k\ge3$ are also given in~\cite{CCPSS}.}

\begin{theorem}\label{thm:main}
Let $m$ and $n$ be integers such that $2 \le m \le n$. Then
\[
\gamma(F_2(K_{m,n}))=
\begin{cases}
\left\lceil\dfrac m2\right\rceil+n-2 & \text{if } m=3, \text{ or } m\ge 5 \text{ is odd and } n=m+1,\\[8pt]
\left\lceil\dfrac m2\right\rceil+n-1 & \text{otherwise}.
\end{cases}
\]
\end{theorem}

In particular, $\gamma(F_2(K_{m,n}))=\frac m2+n-1$ for every even $m$, and $\gamma(F_2(K_{3,n}))=n$ for every $n\ge3$.

Theorem~\ref{thm:main} shows that $\gamma(F_2(K_{m,n}))$ is given by the single formula $\lceil m/2\rceil+n-1$, with exactly two exceptions in which the domination number drops by one\rev{: $m=3$ (for every $n\ge3$), and $m\ge5$ odd with $n=m+1$ (see Subsection~\ref{sec:odd-exception} and Remark~\ref{rem:m3}).}

All our arguments rely on an explicit characterization of the dominating sets of $F_2(K_{m,n})$ (Lemma~\ref{lem:char}), which translates the domination of $F_2(K_{m,n})$ into a problem about a spanning subgraph of $K_{m,n}$ and two families of pairs inside the parts of the bipartition. We expect this approach to be useful for other graphs whose $2$-token graph has a similar structure, such as complete multipartite graphs (see~\cite{CR-multi}).

The rest of the paper is organized as follows. Section~\ref{sec:lemmas} fixes notation, gives the characterization of the dominating sets of $F_2(K_{m,n})$ (Lemma~\ref{lem:char}), and derives several auxiliary lemmas from it. Theorem~\ref{thm:main} is proved in Section~\ref{sec:even} for even $m$, in Section~\ref{sec:small} for $m=3$, and in Section~\ref{sec:odd} for odd $m\ge5$; the latter is divided into the exceptional case $n=m+1$ (Subsection~\ref{sec:odd-exception}) and the general case $n\ne m+1$ (Subsection~\ref{sec:odd-general}). In Section~\ref{sec:computation} we report a computational verification of Theorem~\ref{thm:main} for small values of $m$ and $n$, and Section~\ref{sec:conclusion} contains some concluding remarks.

\section{Auxiliary lemmas}\label{sec:lemmas}
Before stating our auxiliary lemmas, we introduce the notation and partition of $F_2(K_{m,n})$ that will be used throughout the paper.

Let $G=K_{m,n}$ with bipartition $A=\{a_1,\dots,a_m\}$, $B=\{b_1,\dots,b_n\}$, $m\le n$. The vertices of $F_2(K_{m,n})$ split into three classes:
\[
\mathcal A=\binom A2,\qquad \mathcal B=\binom B2,\qquad \mathcal X=\{\{a,b\}: a\in A,\ b\in B\}.
\]
Every edge of $F_2(K_{m,n})$ joins a vertex of $\mathcal X$ with a vertex of $\mathcal A\cup\mathcal B$; in particular, $F_2(K_{m,n})$ is bipartite with parts $\mathcal X$ and $\mathcal A\cup\mathcal B$ (see~\cite{dealba2020} for more details).

For a set $D\subseteq V(F_2(K_{m,n}))$, write $D_\mathcal A:=D\cap\mathcal A$, $D_\mathcal B:=D\cap\mathcal B$ and $D_\mathcal X:=D\cap\mathcal X$. Since every vertex of $F_2(K_{m,n})$ is a pair of vertices of $K_{m,n}$, it is convenient to regard $D_\mathcal A$, $D_\mathcal B$ and $D_\mathcal X$ as sets of edges of the complete graph on $A$, the complete graph on $B$, and $K_{m,n}$, respectively. Let $H:=(A\cup B,\,D_\mathcal X)$ denote the spanning subgraph of $K_{m,n}$ with edge set $D_\mathcal X$, and for $v\in A\cup B$ let $d_H(v)$ be the degree of $v$ in $H$, that is, the number of elements of $D_\mathcal X$ containing $v$. Let
\[
A_0:=\{a\in A: d_H(a)=0\},\qquad B_0:=\{b\in B: d_H(b)=0\},\qquad B_1:=\{b\in B: d_H(b)=1\}.
\]
Thus $A_0$ and $B_0$ are the sets of isolated vertices of $H$ lying in $A$ and in $B$, respectively; we set $k_A:=|A_0|$ and $k_B:=|B_0|$. Let
\[
C_\mathcal A:=\bigcup_{\{a,a'\}\in D_\mathcal A}\{a,a'\}\subseteq A,\qquad C_\mathcal B:=\bigcup_{\{b,b'\}\in D_\mathcal B}\{b,b'\}\subseteq B
\]
be the sets of vertices \emph{covered} by $D_\mathcal A$ and $D_\mathcal B$, respectively, and let $\overline{C}_\mathcal A:=A\setminus C_\mathcal A$ and $\overline{C}_\mathcal B:=B\setminus C_\mathcal B$ denote the sets of \emph{uncovered} vertices. Finally, throughout the paper, we use the following parameters
\[
x:=|D_\mathcal X|,\qquad y:=|D_\mathcal A|,\qquad z:=|D_\mathcal B|,\qquad s:=|\overline{C}_\mathcal A|,\qquad t:=|\overline{C}_\mathcal B|,
\]
so that $|D|=x+y+z$. We use these abbreviations mainly in some of the proofs; in the statements of our results we write the corresponding cardinalities explicitly. Since each pair covers at most two vertices, we have
\begin{equation}\label{eq:st}
|\overline{C}_\mathcal A|\ge m-2|D_\mathcal A|\qquad\text{and}\qquad |\overline{C}_\mathcal B|\ge n-2|D_\mathcal B|.
\end{equation}

\begin{example}\label{ex:H}
Let $m=3$, $n=4$, $A=\{a_1,a_2,a_3\}$ and $B=\{b_1,b_2,b_3,b_4\}$, and consider
\[
D=\big\{\{a_1,a_2\},\,\{b_1,b_2\},\,\{a_3,b_3\},\,\{a_3,b_4\}\big\}.
\]
Then $D_\mathcal A=\{\{a_1,a_2\}\}$, $D_\mathcal B=\{\{b_1,b_2\}\}$ and $D_\mathcal X=\{\{a_3,b_3\},\{a_3,b_4\}\}$, so $H$ is the spanning subgraph of $K_{3,4}$ with the two edges $a_3b_3$ and $a_3b_4$ (see Figure~\ref{fig:H}). Hence $d_H(a_3)=2$, $d_H(b_3)=d_H(b_4)=1$, and all other vertices are isolated in $H$; thus
\[
A_0=\{a_1,a_2\},\quad B_0=\{b_1,b_2\},\quad k_A=k_B=2.
\]
Moreover, $C_\mathcal A=\{a_1,a_2\}$, $\overline{C}_\mathcal A=\{a_3\}$, $C_\mathcal B=\{b_1,b_2\}$ and $\overline{C}_\mathcal B=\{b_3,b_4\}$, so that $x=2$, $y=z=1$, $s=1$ and $t=2$. Note that the conditions of Lemma~\ref{lem:char} below are satisfied: $\binom{A_0}2=\{\{a_1,a_2\}\}\subseteq D_\mathcal A$, $\binom{B_0}2=\{\{b_1,b_2\}\}\subseteq D_\mathcal B$, and $\{a_3,b_3\},\{a_3,b_4\}\in D_\mathcal X$. Hence $D$ is a dominating set of $F_2(K_{3,4})$ of size $4$; by Proposition~\ref{prop:m3}, it is in fact a minimum one.
\end{example}

\begin{figure}[ht]
\centering
\begin{tikzpicture}[
  every node/.style={font=\small},
  filled/.style={circle,draw,fill=black,inner sep=1.6pt},
  hollow/.style={circle,draw,fill=white,inner sep=1.6pt},
  hedge/.style={very thick},
  pair/.style={dashed,red,thick}]
  \node[hollow,label=above:$a_1$] (a1) at (1,2) {};
  \node[hollow,label=above:$a_2$] (a2) at (2.5,2) {};
  \node[filled,label=above:$a_3$] (a3) at (4,2) {};
  \node[hollow,label=below:$b_1$] (b1) at (0.25,0) {};
  \node[hollow,label=below:$b_2$] (b2) at (1.75,0) {};
  \node[filled,label=below:$b_3$] (b3) at (3.25,0) {};
  \node[filled,label=below:$b_4$] (b4) at (4.75,0) {};
  \draw[hedge] (a3) -- (b3);
  \draw[hedge] (a3) -- (b4);
  \draw[pair] (a1) to[bend right=35] (a2);
  \draw[pair] (b1) to[bend left=35] (b2);
  \node[anchor=west] at (5.6,2) {$A$};
  \node[anchor=west] at (5.6,0) {$B$};
\end{tikzpicture}
\caption{The graph $H$ of Example~\ref{ex:H}: its edges (solid) are the elements of $D_\mathcal X$, and its isolated vertices (white) form $A_0\cup B_0$. The dashed curves represent the pairs of $D_\mathcal A$ and $D_\mathcal B$, which are not edges of $H$.}
\label{fig:H}
\end{figure}
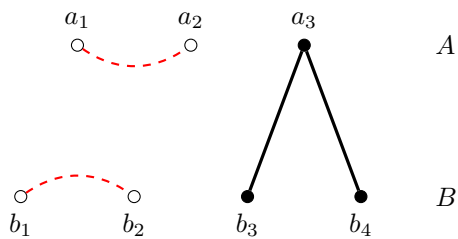

The following characterization is the basis of all our arguments about dominating sets.

\begin{lemma}\label{lem:char}
A set $D\subseteq V(F_2(K_{m,n}))$ is a dominating set of $F_2(K_{m,n})$ if and only if the following three conditions hold:
\begin{enumerate}[label=\textup{(\roman*)}]
\item $\binom{A_0}{2}\subseteq D_\mathcal A$;
\item $\binom{B_0}{2}\subseteq D_\mathcal B$;
\item $\{a,b\}\in D_\mathcal X$ for every $a\in\overline{C}_\mathcal A$ and every $b\in\overline{C}_\mathcal B$.
\end{enumerate}
\end{lemma}
\begin{proof}
Two vertices of $F_2(K_{m,n})$ are adjacent if and only if they share exactly one element and the two remaining elements lie on different sides of the bipartition $\{A,B\}$. Hence:
\begin{itemize}
\item the neighbors of $\{a,a'\}\in\mathcal A$ are the pairs $\{a,b\}$ and $\{a',b\}$ with $b\in B$;
\item the neighbors of $\{b,b'\}\in\mathcal B$ are the pairs $\{a,b\}$ and $\{a,b'\}$ with $a\in A$;
\item the neighbors of $\{a,b\}\in\mathcal X$ are the pairs $\{a,a'\}$ with $a'\in A\setminus\{a\}$ and the pairs $\{b,b'\}$ with $b'\in B\setminus\{b\}$.
\end{itemize}
Consequently, $\{a,a'\}\in\mathcal A$ is dominated by $D$ if and only if $\{a,a'\}\in D_\mathcal A$ or $d_H(a)+d_H(a')\ge1$, that is, if and only if $\{a,a'\}\in D_\mathcal A$ or $\{a,a'\}\not\subseteq A_0$. Hence every vertex of $\mathcal A$ is dominated if and only if (i) holds, and, symmetrically, every vertex of $\mathcal B$ is dominated if and only if (ii) holds. Finally, $\{a,b\}\in\mathcal X$ is dominated by $D$ if and only if $\{a,b\}\in D_\mathcal X$, or $a\in C_\mathcal A$, or $b\in C_\mathcal B$. Hence every vertex of $\mathcal X$ is dominated if and only if (iii) holds.
\end{proof}

Condition~(iii) says that, in the graph $H$, every vertex of $\overline{C}_\mathcal B$ is adjacent to every vertex of $\overline{C}_\mathcal A$. In particular, if $D$ is a dominating set, then
\begin{equation}\label{eq:forced}
d_H(b)\ge|\overline{C}_\mathcal A|\quad\text{for every } b\in\overline{C}_\mathcal B,\qquad\text{and}\qquad |D_\mathcal X|\ge|\overline{C}_\mathcal A|\cdot|\overline{C}_\mathcal B|.
\end{equation}

We now derive some consequences of Lemma~\ref{lem:char}.

\begin{lemma}\label{lem:basic}
Let $D$ be a dominating set of $F_2(K_{m,n})$. Then
\[
|D_\mathcal X|+|D_\mathcal B|\ge n-1.
\]
Moreover, if equality holds, then $k_B\in\{1,2\}$, $|D_\mathcal X|=n-k_B$ and $|D_\mathcal B|=\binom{k_B}2$.
\end{lemma}
\begin{proof}
By Lemma~\ref{lem:char}(ii), $\binom{B_0}{2}\subseteq D_\mathcal B$, so $|D_\mathcal B| \ge\binom{k_B}2$. On the other hand, every element of $D_\mathcal X$ contains exactly one vertex of $B$, and every $b\in B\setminus B_0$ satisfies $d_H(b)\ge 1$; hence
\[
|D_\mathcal X|=\sum_{b\in B} d_H(b) \ge n-k_B.
\]
Adding these two bounds yields $|D_\mathcal X|+|D_\mathcal B|\ge \lambda(k_B):=n-k_B+\binom{k_B}2$. Since $\binom k2-k=\frac{k(k-3)}2$ equals $-1$ for $k\in\{1,2\}$ and is non-negative for every other non-negative integer $k$, we get $|D_\mathcal X|+|D_\mathcal B|\ge n-1$, and equality yields $k_B\in\{1,2\}$ together with equality in both bounds.
\end{proof}

\begin{lemma}\label{lem:forcedA}
Let $D$ be a dominating set of $F_2(K_{m,n})$. If $\overline{C}_\mathcal A\neq\emptyset$, then $B_0\subseteq C_\mathcal B$; consequently, $|D_\mathcal B|\ge\lceil k_B/2\rceil$.
\end{lemma}
\begin{proof}
Let $a^*\in\overline{C}_\mathcal A$ and $b\in B_0$. Since $d_H(b)=0$, we have $\{a^*,b\}\notin D_\mathcal X$. Thus, by Lemma~\ref{lem:char}(iii), we have $b\notin\overline{C}_\mathcal B$, that is, $b\in C_\mathcal B$. Hence $B_0\subseteq C_\mathcal B$, which implies $k_B\le |C_\mathcal B|\le 2|D_\mathcal B|$.
\end{proof}

\begin{lemma}\label{lem:symmetric}
Let $D$ be a dominating set of $F_2(K_{m,n})$. Then 
\[
|D_\mathcal X| + |D_\mathcal A|\ge m-1.
\]
\end{lemma}
\begin{proof}
By Lemma~\ref{lem:char}(i), $\binom{A_0}{2}\subseteq D_\mathcal A$, so $\binom{k_A}2\le y$; and, as in the proof of Lemma~\ref{lem:basic}, $x=\sum_{a\in A}d_H(a)\ge m-k_A$. If $k_A\ge y+2$, then $\binom{k_A}2\ge\binom{y+2}2=\frac{(y+2)(y+1)}2>y$, a contradiction. Hence $k_A\le y+1$, and $x\ge m-k_A\ge m-y-1$.
\end{proof}

\begin{lemma}\label{lem:zeroz}
Let $D$ be a dominating set of $F_2(K_{m,n})$ with $D_\mathcal A=\emptyset$ and $\overline{C}_\mathcal B\neq\emptyset$. Then $|D|\ge n+m-2$.
\end{lemma}
\begin{proof}
Since $D_\mathcal A=\emptyset$, we have $y=0$ and $\overline{C}_\mathcal A=A$, so $s=m$, and Lemma~\ref{lem:char}(iii) gives $d_H(b)=m$ for every $b\in\overline{C}_\mathcal B$. Since $\overline{C}_\mathcal A\neq\emptyset$, Lemma~\ref{lem:forcedA} gives $B_0\subseteq C_\mathcal B$ and $z\ge\lceil k_B/2\rceil$; in particular, $B_0$ and $\overline{C}_\mathcal B$ are disjoint. Moreover, $z\ge\binom{k_B}2$ by Lemma~\ref{lem:char}(ii). Summing $d_H(b)$ over $b\in B$, we obtain
\begin{align*}
x &=\sum_{b \in \overline{C}_\mathcal B} d_H(b)+\sum_{b \in C_\mathcal B} d_H(b),\\
x &\ge mt+(n-t-k_B),
\end{align*}
and therefore
\[
|D|=x+z\ge n+(m-1)t-k_B+\theta(k_B),\qquad\text{where}\quad \theta(k):=\max\Big\{\binom k2,\Big\lceil\frac k2\Big\rceil\Big\}.
\]
A direct check shows that $\theta(k)-k\ge-1$ for every integer $k\ge0$ (with equality only for $k=2$). Since $\overline{C}_\mathcal B\neq\emptyset$, we have $t\ge1$, and we conclude that $|D|\ge n+(m-1)-1=n+m-2$.
\end{proof}

The following arithmetic lemma is used in Section~\ref{sec:odd} to bound $|D_\mathcal X|+|D_\mathcal B|$ when every vertex of $B$ is covered by $D_\mathcal B$; in that situation the function $\psi(k_B)$ in the following lemma is a lower bound for $|D_\mathcal B|$ (Remark~\ref{rem:hbound}).

\begin{lemma}\label{lem:gmono}
Fix integers $n\ge 1$ and $0\le P < n$, and for $k\ge 0$ define
\[
\varphi(k):=\max\{P,n-k\}+\psi(k),\quad \text{where}\quad \psi(k):=\begin{cases}
\left\lceil\dfrac n2\right\rceil & \text{if } k\le 1,\\[6pt]
\dbinom k2+\left\lceil\dfrac{n-k}2\right\rceil & \text{if } k\ge 2.
\end{cases}
\]
Then $\varphi(1)=\varphi(0)-1$, and $\varphi(k)$ is non-decreasing for $k\ge 2$. Furthermore, if $P\le n-2$, then $\varphi(2)=\varphi(1)-1$ and, consequently, $\min_{k\ge 0} \{\varphi(k)\} = \varphi(2)$.
\end{lemma}
\begin{proof}
Since $P<n$, we have $\max\{P,n\}=n$ and $\max\{P,n-1\}=n-1$. Thus, $\varphi(0)=n+\lceil n/2\rceil$ and $\varphi(1)=(n-1)+\lceil n/2\rceil=\varphi(0)-1$. If $P\le n-2$, then $\max\{P,n-2\}=n-2$, and evaluating at $k=2$ yields
\[
\varphi(2)-\varphi(1)=\left((n-2)+1+\left\lceil\frac{n-2}{2}\right\rceil\right)-\left((n-1)+\left\lceil\frac n2\right\rceil\right)=\left\lceil\frac{n-2}{2}\right\rceil-\left\lceil\frac n2\right\rceil=-1,
\]
which proves that $\varphi(2)=\varphi(1)-1$.

For any $k\ge 2$, the discrete difference $\varphi(k+1)-\varphi(k)$ expands to
\[
\left(\max\{P,n-k-1\}-\max\{P, n-k\}\right)+\left(\binom{k+1}2-\binom k2\right)+\left(\left\lceil\frac{n-k-1}2\right\rceil-\left\lceil\frac{n-k}2\right\rceil\right).
\]
The first term is either $0$ or $-1$, the second term equals $k$, and the third term is either $0$ or $-1$. Therefore, 
\[
\varphi(k+1)-\varphi(k)\ge -1 + k - 1 = k-2\ge 0 \quad \text{for all } k\ge 2,
\]
which shows that $\varphi$ is non-decreasing on $[2,\infty)$. The statement regarding the global minimum follows immediately.
\end{proof}

\begin{remark}\label{rem:hbound}
Let $D$ be a dominating set of $F_2(K_{m,n})$ with $\overline{C}_\mathcal B=\emptyset$, that is, $t=0$. Then $|D_\mathcal B|\ge \psi(k_B)$, with $\psi$ as in Lemma~\ref{lem:gmono}. Indeed, since $t=0$, we have $C_\mathcal B=B$, that is, the pairs of $D_\mathcal B$ cover all $n$ vertices of $B$. Each pair covers two vertices, so $|D_\mathcal B|\ge\lceil n/2\rceil$, which is the claim when $k_B\le1$. If $k_B\ge2$, then $\binom{B_0}{2}\subseteq D_\mathcal B$ by Lemma~\ref{lem:char}(ii), and these $\binom{k_B}2$ pairs cover no vertex of $B\setminus B_0$. The $n-k_B$ vertices of $B\setminus B_0$ are therefore covered by other pairs of $D_\mathcal B$, each covering at most two of them, so $|D_\mathcal B|\ge\binom{k_B}2+\lceil (n-k_B)/2\rceil$.
\end{remark}
\begin{corollary}[Value of $\varphi(2)$]\label{cor:g2}
Let $m$ be an odd integer, $n\ge m$, and set $r:=n-m+1$. With $\varphi$ defined as in Lemma~\ref{lem:gmono} for this $n$ and any $P\le n-2$, we have
\[
\varphi(2)=\frac{m-1}2+n+\begin{cases}
\dfrac{r-4}2 & \text{if } r \text{ is even},\\[6pt]
\dfrac{r-3}2 & \text{if } r \text{ is odd}.
\end{cases}
\]
In particular, $\varphi(2)\ge\frac{m-1}2+n-1$, with equality if and only if $r\in\{1,2\}$, and $\varphi(2)\ge\frac{m-1}2+n$ for all $r\ge 3$.
\end{corollary}
\begin{proof}
Since $P\le n-2$, we have $\max\{P,n-2\}=n-2$, so $\varphi(2)=(n-2)+\binom22+\left\lceil\frac{n-2}{2}\right\rceil=(n-1)+\left\lceil\frac{n-2}{2}\right\rceil$. Since $n=m-1+r$ and $m-1$ is even, we can rewrite the ceiling term as
\[
\left\lceil\frac{n-2}{2}\right\rceil = \left\lceil\frac{(m-1)+(r-2)}{2}\right\rceil = \frac{m-1}{2} + \left\lceil\frac{r-2}{2}\right\rceil.
\]
Substituting this back into $\varphi(2)$ gives
\[
\varphi(2) = (n-1) + \frac{m-1}{2} + \left\lceil\frac{r-2}{2}\right\rceil = \frac{m-1}{2} + n + \left( \left\lceil\frac{r-2}{2}\right\rceil - 1 \right).
\]
If $r$ is even, $\left\lceil\frac{r-2}{2}\right\rceil - 1 = \frac{r-2}{2} - 1 = \frac{r-4}{2}$. If $r$ is odd, $\left\lceil\frac{r-2}{2}\right\rceil - 1 = \frac{r-1}{2} - 1 = \frac{r-3}{2}$.

Finally, for integers $r \ge 1$, the correction term equals $-1$ if $r \in \{1,2\}$ (namely, for $r=1$ odd and $r=2$ even) and is non-negative for all $r \ge 3$.
\end{proof}

\section{Proof of Theorem~\ref{thm:main} for even $m$}\label{sec:even}

Throughout this section, $m$ is even and $2\le m\le n$. We prove that $\gamma(F_2(K_{m,n}))=\frac m2+n-1$.

\subsection{Upper bound}

Let $A=\{a_1,\dots,a_m\}$, $B=\{b_1,\dots,b_n\}$. Define the internal matching
\[
\mathcal M(A)=\{\{a_1,a_2\},\{a_3,a_4\},\dots,\{a_{m-1},a_m\}\},\qquad |\mathcal M(A)|=\frac m2,
\]
which covers all of $A$, together with
\[
\mathcal S=\{\{a_i,b_i\}: 1\le i\le m-1\}\cup\{\{a_m,b_i\}: m\le i\le n-1\},\qquad |\mathcal S|=n-1.
\]
Let $D=\mathcal M(A)\cup\mathcal S$, so $|D|=\frac m2+n-1$.

\begin{proposition}[Upper bound for even $m$]\label{prop:even-upper}
For $m$ even, $2\le m\le n$, the set $D$ is a dominating set of $F_2(K_{m,n})$.
\end{proposition}
\begin{proof}
We check the three conditions of Lemma~\ref{lem:char}. Every $a_i$ with $1\leq i\le m-1$ lies in $\{a_i,b_i\}\in\mathcal S$, so $A_0\subseteq\{a_m\}$ and (i) holds trivially. Every $b_i$ with $1\leq i\le n-1$ lies in an element of $\mathcal S$, so $B_0\subseteq\{b_n\}$ and (ii) holds trivially. Finally, $\mathcal M(A)$ covers all of $A$, so $\overline{C}_\mathcal A=\emptyset$ and (iii) holds vacuously.
\end{proof}

\subsection{Lower bound}

\begin{proposition}[Lower bound for even $m$]\label{prop:even-lower}
For $m$ even, $2\le m\le n$, $\gamma(F_2(K_{m,n}))\ge\frac m2+n-1$.
\end{proposition}
\begin{proof}
Suppose, for contradiction, that some dominating set $D$ satisfies $|D|\le\frac m2+n-2$. By Lemma~\ref{lem:basic}, $x+z\ge n-1$, so
\[
y\le\Big(\frac m2+n-2\Big)-(n-1)=\frac m2-1.
\]
Since each pair of $D_\mathcal A$ covers at most $2$ vertices, $D_\mathcal A$ covers at most $m-2$ vertices of $A$, so at least two vertices of $A$ are uncovered by $D_\mathcal A$, that is, $|\overline{C}_\mathcal A|\geq 2$.

Recall that $\binom{B_0}2\subseteq D_\mathcal B$, by Lemma~\ref{lem:char}(ii). We have two cases depending on the intersection of $B_1$ with the uncovered set $\overline{C}_\mathcal B$. 

\paragraph{Case 1: $B_1 \cap \overline{C}_\mathcal B \neq \emptyset$.} Let $b^*\in B_1\cap\overline{C}_\mathcal B$, and let $\{a^*,b^*\}$ be the unique element of $D_\mathcal X$ containing $b^*$. Since $b^*\in\overline{C}_\mathcal B$, Lemma~\ref{lem:char}(iii) gives $\{a,b^*\}\in D_\mathcal X$ for every $a\in\overline{C}_\mathcal A$; hence $\overline{C}_\mathcal A\subseteq\{a^*\}$, contradicting $|\overline{C}_\mathcal A|\ge2$.

\paragraph{Case 2: $B_1 \cap \overline{C}_\mathcal B = \emptyset$.}
Since $\overline{C}_\mathcal A\neq\emptyset$, Lemma~\ref{lem:forcedA} gives $B_0\subseteq C_\mathcal B$. Hence $D_\mathcal B$ covers both $B_0$ and $B_1$, so
\[
z\ge\binom{k_B}2+\left\lceil\frac{|B_1|}2\right\rceil\ge\frac{k_B(k_B-1)}2+\frac{|B_1|}2.
\]
On the other hand, the remaining vertices of $B$ satisfy $d_H(b)\ge2$, so
\[
x\ge0\cdot k_B+1\cdot|B_1|+2(n-k_B-|B_1|)=2n-2k_B-|B_1|.
\]
Summing, and using $|B_1|\le n-k_B$:
\[
|D|\ge x+z\ge\frac32n+\frac{k_B^2-4k_B}2.
\]
We check the possible values of $k_B$:
\begin{itemize}
\item $k_B=0$: $|D|\ge\frac32n\ge n+\frac m2>\frac m2+n-2$, a contradiction.
\item $k_B=1$ or $k_B\ge3$: $|D|\ge\frac32n-\frac32$; combined with $|D|\le\frac m2+n-2$ this yields $n\le m-1$, impossible since $n\ge m$.
\item $k_B=2$: $|D|\ge\frac32n-2$. If $n>m$, this already exceeds $\frac m2+n-2$. If $n=m$, the hypothesis $|D|\le\frac32m-2$ together with $|D|\ge\frac32m-2$
gives $|D|=\frac32m-2$, so every inequality used to obtain this bound is an
equality: $y=0$, $|B_1|=n-k_B=m-2$, and every vertex of $B\setminus B_0$ lies
in $B_1$, whence $x=|B_1|=m-2$. But then Lemma~\ref{lem:symmetric} (with $y=0$) gives $x\ge m-1$, a contradiction.
\end{itemize}
Every case yields a contradiction, so $|D|\ge\frac m2+n-1$.
\end{proof}

Propositions~\ref{prop:even-upper} and~\ref{prop:even-lower} prove Theorem~\ref{thm:main} for even $m$.

\section{Proof of Theorem~\ref{thm:main} for $m=3$}\label{sec:small}

In this section we prove the case $m=3$ of Theorem~\ref{thm:main}.

\begin{proposition}\label{prop:m3}
For every $n\ge3$, $\gamma(F_2(K_{3,n}))=n$.
\end{proposition}
\begin{proof}
Let $A=\{a_1,a_2,a_3\}$. The set $D=\{\{a_1,a_2\}\}\cup\{\{b_1,b_2\}\}\cup\{\{a_3,b_j\}:3\le j\le n\}$ has size $n$, and it is dominating by Lemma~\ref{lem:char}: $A_0=\{a_1,a_2\}$ and $B_0=\{b_1,b_2\}$, whose pairs belong to $D$, and $\overline{C}_\mathcal A=\{a_3\}$, $\overline{C}_\mathcal B=\{b_3,\dots,b_n\}$, so (iii) holds by construction.

For the lower bound, suppose that $D$ is a dominating set with $|D|\le n-1$. Since $x+z\ge n-1$ by Lemma~\ref{lem:basic}, we get $y=0$ and $x+z=n-1$; hence, again by Lemma~\ref{lem:basic}, $k_B\in\{1,2\}$ and $z=\binom{k_B}2\le1$. Recall that $|\overline{C}_\mathcal B|\ge n-2z$ by~\eqref{eq:st}. Since $D_\mathcal A=\emptyset$, we have $\overline{C}_\mathcal A=A$ and hence for every $b\in \overline{C}_\mathcal B$ we have $\{\{a_1,b\},\{a_2,b\},\{a_3,b\}\} \subseteq D_\mathcal X$ (Lemma~\ref{lem:char}(iii)), giving $x\ge 3|\overline{C}_\mathcal B|\ge3(n-2z)$; combined with $x=n-1-z$ this yields $5z\ge 2n+1$. Since $n\ge3$, we get $z\ge\frac{2n+1}{5}>1$, and since $z$ is an integer, $z\ge2$, contradicting $z\le1$.
\end{proof}

Proposition~\ref{prop:m3} proves Theorem~\ref{thm:main} for $m=3$.

\section{Proof of Theorem~\ref{thm:main} for odd $m\ge5$}\label{sec:odd}

Throughout this section, $m$ is odd, $m\ge5$, $n\ge m$. Write $M_0:=\lceil \frac{m}{2}\rceil+n-1=\frac{m-1}2+n$.

Figure~\ref{fig:constructions} shows, for $m=5$, the two constructions that give the upper bounds of this section.

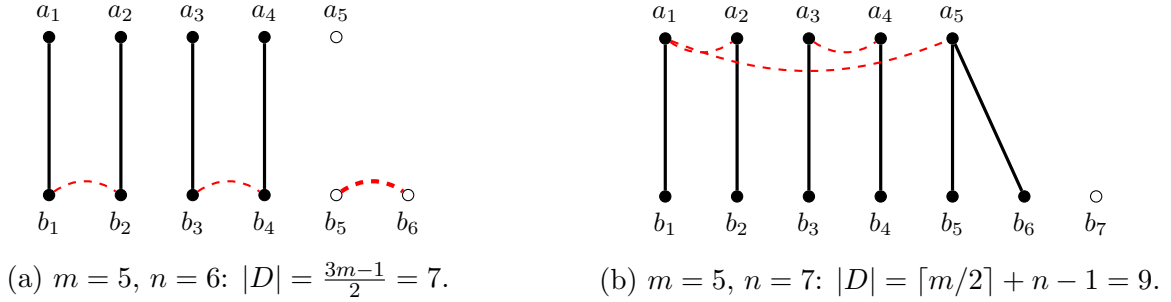
\begin{figure}[ht]
\centering
\begin{minipage}{0.48\textwidth}\centering
\begin{tikzpicture}[scale=0.95,
  every node/.style={font=\small},
  filled/.style={circle,draw,fill=black,inner sep=1.5pt},
  hollow/.style={circle,draw,fill=white,inner sep=1.5pt},
  hedge/.style={very thick},
  pair/.style={dashed,red,thick}]
  \node[filled,label=above:$a_{1}$] (a1) at (0.00,2.2) {};
  \node[filled,label=above:$a_{2}$] (a2) at (1.00,2.2) {};
  \node[filled,label=above:$a_{3}$] (a3) at (2.00,2.2) {};
  \node[filled,label=above:$a_{4}$] (a4) at (3.00,2.2) {};
  \node[hollow,label=above:$a_{5}$] (a5) at (4.00,2.2) {};
  \node[filled,label=below:$b_{1}$] (b1) at (0.00,0) {};
  \node[filled,label=below:$b_{2}$] (b2) at (1.00,0) {};
  \node[filled,label=below:$b_{3}$] (b3) at (2.00,0) {};
  \node[filled,label=below:$b_{4}$] (b4) at (3.00,0) {};
  \node[hollow,label=below:$b_{5}$] (b5) at (4.00,0) {};
  \node[hollow,label=below:$b_{6}$] (b6) at (5.00,0) {};
  \draw[hedge] (a1) -- (b1);
  \draw[hedge] (a2) -- (b2);
  \draw[hedge] (a3) -- (b3);
  \draw[hedge] (a4) -- (b4);
  \draw[pair] (b1) to[bend left=35] (b2);
  \draw[pair] (b3) to[bend left=35] (b4);
  \draw[pair,line width=1.6pt] (b5) to[bend left=35] (b6);
\end{tikzpicture}\\[1mm]
(a) $m=5$, $n=6$: $|D|=\frac{3m-1}2=7$.
\end{minipage}\hfill
\begin{minipage}{0.48\textwidth}\centering
\begin{tikzpicture}[scale=0.95,
  every node/.style={font=\small},
  filled/.style={circle,draw,fill=black,inner sep=1.5pt},
  hollow/.style={circle,draw,fill=white,inner sep=1.5pt},
  hedge/.style={very thick},
  pair/.style={dashed,red,thick}]
  \node[filled,label=above:$a_{1}$] (a1) at (0.00,2.2) {};
  \node[filled,label=above:$a_{2}$] (a2) at (1.00,2.2) {};
  \node[filled,label=above:$a_{3}$] (a3) at (2.00,2.2) {};
  \node[filled,label=above:$a_{4}$] (a4) at (3.00,2.2) {};
  \node[filled,label=above:$a_{5}$] (a5) at (4.00,2.2) {};
  \node[filled,label=below:$b_{1}$] (b1) at (0.00,0) {};
  \node[filled,label=below:$b_{2}$] (b2) at (1.00,0) {};
  \node[filled,label=below:$b_{3}$] (b3) at (2.00,0) {};
  \node[filled,label=below:$b_{4}$] (b4) at (3.00,0) {};
  \node[filled,label=below:$b_{5}$] (b5) at (4.00,0) {};
  \node[filled,label=below:$b_{6}$] (b6) at (5.00,0) {};
  \node[hollow,label=below:$b_{7}$] (b7) at (6.00,0) {};
  \draw[hedge] (a1) -- (b1);
  \draw[hedge] (a2) -- (b2);
  \draw[hedge] (a3) -- (b3);
  \draw[hedge] (a4) -- (b4);
  \draw[hedge] (a5) -- (b5);
  \draw[hedge] (a5) -- (b6);
  \draw[pair] (a1) to[bend right=35] (a2);
  \draw[pair] (a3) to[bend right=35] (a4);
  \draw[pair] (a1) to[bend right=22] (a5);
\end{tikzpicture}\\[1mm]
(b) $m=5$, $n=7$: $|D|=\lceil m/2\rceil+n-1=9$.
\end{minipage}

\caption{The upper-bound constructions of Section~\ref{sec:odd} for $m=5$, with the conventions of Figure~\ref{fig:H}. (a)~The construction of Section~\ref{sec:odd-exception} for $n=m+1$; here $A_0=\{a_5\}$, $B_0=\{b_5,b_6\}$ and $\overline{C}_\mathcal B=\emptyset$. (b)~The construction $\mathcal M(A)\cup\mathcal S$ of Section~\ref{sec:odd-general}; here $A_0=\emptyset$, $B_0=\{b_7\}$ and $\overline{C}_\mathcal A=\emptyset$.}
\label{fig:constructions}
\end{figure}
First, we prove the following lemma, which deals with the case $1\le|D_\mathcal A|\le\frac{m-3}2$.

\begin{lemma}\label{lem:casec}
Let $D$ be a dominating set of $F_2(K_{m,n})$ such that $1\le|D_\mathcal A|\le\frac{m-3}2$. Then $|D|\ge M_0$, for every $n\ge m$.
\end{lemma}
\begin{proof}
By~\eqref{eq:st}, $s\ge m-2y\ge3$. Set $P:=m-y-1$, so $x\ge P$ by Lemma~\ref{lem:symmetric}.

\paragraph{Case $t\ge1$.} Since $s\ge3$, we have $\overline{C}_\mathcal A\ne\emptyset$, so $B_0\subseteq C_\mathcal B$ by Lemma~\ref{lem:forcedA}; in particular $\overline{C}_\mathcal B\cap B_0=\emptyset$. By Lemma~\ref{lem:char}(iii), each of the $t$ vertices of $\overline{C}_\mathcal B$ has $d_H(b)\ge s$, while every other $b\notin B_0$ has $d_H(b)\ge1$; summing over $B$,
\begin{equation}\label{eqA}
x\ge st+1\cdot((n-t)-k_B)+0\cdot k_B=n-k_B+(s-1)t.
\end{equation}
Using only $z\ge\binom{k_B}2\ge0$ and $\binom{k_B}2-k_B\ge-1$ for all $k_B\ge0$,
\[
x+y+z\ge y+n-k_B+(s-1)t+\binom{k_B}2\ge y+n+(s-1)t-1.
\]
Since $t\ge1$ and $s\ge3$, $(s-1)t\ge s-1\ge m-2y-1$, so
\[
x+y+z\ge y+n+(m-2y-1)-1=n+m-y-2.
\]

By hypothesis, $y\le\frac{m-3}2$ and hence $m-y\ge \frac{m+3}2$. Therefore
 \[
 x+y+z\ge n+m-y-2 \geq n+\frac{m-1}2=M_0.
 \]
\paragraph{Case $t=0$.} As in the proof of Lemma~\ref{lem:basic}, $x=\sum_{b\in B}d_H(b)\ge n-k_B$; since also $x \ge P$, we have $x \ge \max\{P,n-k_B\}$, and $z\ge \psi(k_B)$ by Remark~\ref{rem:hbound}. Since 
\[
0 \le P = m - y - 1 < m \le n \quad \text{and} \quad n - 2 \ge m - 2 \ge m - y - 1 = P \quad (\text{using } y \ge 1),
\]
Lemma~\ref{lem:gmono} applies to $x+z\geq \varphi(k_B)$, where $ \varphi(k_B):=\max\{P,n-k_B\}+\psi(k_B)$, and gives $x+y+z\ge y+\min_{k_B}\{\varphi(k_B)\}=y+\varphi(2)$. By Corollary~\ref{cor:g2}, $\varphi(2)\ge M_0-1$ always, so $x+y+z\ge y+M_0-1\ge M_0$ since $y\ge1$.

Both subcases give $|D|=x+y+z\ge M_0$, and the proof is completed.
\end{proof}

\subsection{The case $n=m+1$}\label{sec:odd-exception}

In this subsection we prove that, for $m\ge5$ odd and $n=m+1$,
\[
\gamma(F_2(K_{m,n}))=\left\lceil\frac m2\right\rceil+n-2=\frac{3m-1}2.
\]
This value is \emph{one less} than $M_0$. The formula $M_0$, although a valid upper bound for all $n\ge m$ (Subsection~\ref{sec:odd-general}), is not optimal at this point.

\subsubsection{Upper bound}

Let $A=\{a_1,\dots,a_m\}$, $B=\{b_1,\dots,b_{m+1}\}$. Define 
\[
D_\mathcal A=\emptyset,\qquad D_\mathcal X=\{\{a_i,b_i\}:1\le i\le m-1\},
\]
\[
D_\mathcal B=\{\{b_1,b_2\},\{b_3,b_4\},\dots,\{b_{m-2},b_{m-1}\}, \{b_m,b_{m+1}\}\},
\]
 and $D=D_\mathcal A\cup D_\mathcal B\cup D_\mathcal X$. Then $|D|=(m-1)+\frac{m+1}2=\frac{3m-1}2$.

\begin{proposition}\label{prop:exception-upper}
The set $D$ is a dominating set of $F_2(K_{m,m+1})$.
\end{proposition}
\begin{proof}
We check the three conditions of Lemma~\ref{lem:char}. Every $a_i$ with $1 \leq i\le m-1$ lies in $\{a_i,b_i\}\in D_\mathcal X$, so $A_0=\{a_m\}$ and (i) holds trivially. Similarly, $B_0=\{b_m,b_{m+1}\}$, and $\{b_m,b_{m+1}\}\in D_\mathcal B$, so (ii) holds. Finally, $D_\mathcal B$ is a perfect matching of $B$, so $\overline{C}_\mathcal B=\emptyset$ and (iii) holds vacuously.
\end{proof}
Hence $\gamma(F_2(K_{m,m+1}))\le\frac{3m-1}2$. This construction is illustrated in Figure~\ref{fig:constructions}(a). Note that the two isolated vertices $b_m,b_{m+1}$ of $H$ in $B$ are handled by a single pair, which is at the same time part of the perfect matching of $B$; this is the source of the saving of one element with respect to $M_0$.

\subsubsection{Lower bound}

\begin{proposition}\label{prop:lowerexception}
For $m \ge5$ odd and $n=m+1$, $\gamma(F_2(K_{m,n}))\ge\frac{3m-1}2$.
\end{proposition}
\begin{proof}
Suppose $D$ is dominating with $|D|\le\frac{3m-1}2-1=:M_1-1$, where $M_1:=\frac{3m-1}2$. Since $n-1=m$, Lemma~\ref{lem:basic} gives $x+z\ge m$.

\paragraph{Case (a): $y\ge\frac{m-1}2$.} Then $|D|\ge y+m\ge\frac{m-1}2+m=M_1$, a contradiction.

\paragraph{Case (b): $1\le y\le\frac{m-3}2$.} This is precisely the hypothesis of Lemma~\ref{lem:casec} (with $n=m+1$), which gives $|D|\ge M_0=M_1+1$, a contradiction.

\paragraph{Case (c): $y=0$.} If $t\ge1$, then Lemma~\ref{lem:zeroz} gives $|D|\ge n+m-2=2m-1\ge\frac{3m-1}2=M_1$, a contradiction. Hence $t=0$, that is, $D_\mathcal B$ covers all of $B$. Then Lemma~\ref{lem:char}(i) (with $D_\mathcal A=\emptyset$) implies $k_A\le1$, giving $x\ge m-1$; and covering all $n=m+1$ vertices of $B$ requires $z\ge\lceil(m+1)/2\rceil=\frac{m+1}2$. Hence $|D|\ge(m-1)+\frac{m+1}2=\frac{3m-1}2=M_1$, a contradiction.

In all cases, we obtain a contradiction. Therefore, $|D|\ge M_1=\frac{3m-1}2$.
\end{proof}

Propositions~\ref{prop:exception-upper} and~\ref{prop:lowerexception} prove Theorem~\ref{thm:main} for $m\ge5$ odd and $n=m+1$.

\subsection{The case $n\ne m+1$}\label{sec:odd-general}

In this subsection we prove that, for $m\ge5$ odd, $n\ge m$ and $n\ne m+1$,
\[
\gamma(F_2(K_{m,n}))=\left\lceil\frac m2\right\rceil+n-1=M_0.
\]

\subsubsection{Upper bound}

The same construction $\mathcal M(A)\cup\mathcal S$ from Section~\ref{sec:even} applies, with $\mathcal S$ left unchanged and the adaptation to $m$ odd occurring solely in $\mathcal M(A)$ via
\[
\mathcal M(A)=\{\{a_1,a_2\},\dots,\{a_{m-2},a_{m-1}\}\}\cup\{\{a_1,a_m\}\},\qquad|\mathcal M(A)|=\left\lceil\frac m2\right\rceil,
\]
gives a dominating set of size $\lceil m/2\rceil+n-1$; the verification is identical to Section~\ref{sec:even}, since $\mathcal M(A)$ still covers all of $A$ (see Figure~\ref{fig:constructions}(b)). This establishes $\gamma(F_2(K_{m,n}))\le M_0$ for every $n\ge m$, $m$ odd.

\subsubsection{Lower bound}

\begin{proposition}\label{prop:general-lower}
Let $m\ge5$ be odd, $n\ge m$, and $n\ne m+1$. Then $\gamma(F_2(K_{m,n}))\ge M_0$.
\end{proposition}
\begin{proof}
Suppose $D$ is dominating with $|D|\le M_0-1$.

\paragraph{Case (a): $y\ge\frac{m+1}2$.} By Lemma~\ref{lem:basic}, $|D|\ge y+(n-1)\ge\frac{m+1}2+n-1=M_0$, a contradiction.

\paragraph{Case (b): $y=0$.} Let $r:=n-m+1\ge1$ ($r\ne2$, as $n\ne m+1$). We distinguish according to the value of $t$.

\emph{If $t\ge1$:} Lemma~\ref{lem:zeroz} gives $|D|\ge n+m-2\ge n+\frac{m-1}2=M_0$, since $m\ge3$.

\emph{If $t=0$:} then $x\ge\max\{m-1,n-k_B\}$ (by Lemma~\ref{lem:symmetric} with $y=0$, and as in the proof of Lemma~\ref{lem:basic}) and $z\ge \psi(k_B)$ by Remark~\ref{rem:hbound}; we apply Lemma~\ref{lem:gmono} with $P=m-1$. First, if $n\ge m+1$ (i.e. $r\ge2$), then $n-2\ge m-1=P$, so Lemma~\ref{lem:gmono} gives $\min_{k_B}\{ \varphi(k_B)\}=\varphi(2)$, and Corollary~\ref{cor:g2} gives $\varphi(2)\ge M_0-1$, with equality only at $r=2$, which is not the case because $n\ne m+1$. So $\varphi(2)\ge M_0$ whenever $r\ge3$. Now, if $n=m$ (i.e. $r=1$), then $n-2=m-2<P$, so $k_B=2$ lies in the $\max=P$ branch: $\varphi(2)=P+\psi(2)=(m-1)+1+\lceil(m-2)/2\rceil=\frac{3m-1}2=M_0$ (using $m$ odd), which equals $\varphi(1)=\varphi(0)-1=m+\lceil m/2\rceil-1=\frac{3m-1}2$ as well. Since $\varphi$ is non-decreasing for $k\ge2$ (Lemma~\ref{lem:gmono}), we get $\min_{k\ge0}\varphi(k)=\min\{\varphi(1),\varphi(2)\}=M_0$. In either case, $|D|=x+z\ge\varphi(k_B)\ge M_0$.

Hence, in both subcases, $|D|\ge M_0$, a contradiction.

\paragraph{Case (c1): $1\le y\le\frac{m-3}2$.} This is precisely the hypothesis of Lemma~\ref{lem:casec}, which gives $|D|\ge M_0$ directly, for every $n\ge m$ (in particular for $n\ne m+1$).

\paragraph{Case (c2): $y=\frac{m-1}2$.} Since $|D|\le M_0-1=y+n-1$ and $x+z\ge n-1$ by Lemma~\ref{lem:basic}, we get $x+z=n-1$. Hence, again by Lemma~\ref{lem:basic}, $k_B\in\{1,2\}$, $x=n-k_B$ and $z=\binom{k_B}2$. Since $s\ge m-2y=1$, we have $\overline{C}_\mathcal A\neq\emptyset$, and Lemma~\ref{lem:forcedA} gives $k_B\le 2z$; this rules out $k_B=1$ (for which $z=0$). Thus $k_B=2$, $z=1$, and by Lemma~\ref{lem:char}(ii), $D_\mathcal B=\{B_0\}$, so that $\overline{C}_\mathcal B=B\setminus B_0$. By Lemma~\ref{lem:char}(iii), every $b\in B\setminus B_0$ satisfies $d_H(b)\ge s$, hence
\[
n-2=x=\sum_{b\in B\setminus B_0}d_H(b)\ge (n-2)s.
\]
Since $n-2>0$, this implies $s=1$ and $d_H(b)=1$ for every $b\in B\setminus B_0$. Writing $\overline{C}_\mathcal A=\{a^*\}$, Lemma~\ref{lem:char}(iii) gives $\{a^*,b\}\in D_\mathcal X$ for every $b\in B\setminus B_0$, so $D_\mathcal X=\{\{a^*,b\}: b\in B\setminus B_0\}$. Consequently $A_0=A\setminus\{a^*\}$, and Lemma~\ref{lem:char}(i) requires $\binom{m-1}2\le y=\frac{m-1}2$, that is, $m\le3$, a contradiction. This configuration is illustrated in Figure~\ref{fig:c2}.

\begin{figure}[ht]
\centering
\begin{tikzpicture}[
  every node/.style={font=\small},
  filled/.style={circle,draw,fill=black,inner sep=1.5pt},
  hollow/.style={circle,draw,fill=white,inner sep=1.5pt},
  hedge/.style={very thick},
  pair/.style={dashed,red,thick}]
  \node[hollow,label=above:$a_{1}$] (a1) at (0.00,2.2) {};
  \node[hollow,label=above:$a_{2}$] (a2) at (1.00,2.2) {};
  \node[hollow,label=above:$a_{3}$] (a3) at (2.00,2.2) {};
  \node[hollow,label=above:$a_{4}$] (a4) at (3.00,2.2) {};
  \node[filled,label=above:$a_{5}$] (a5) at (4.00,2.2) {};
  \node[hollow,label=below:$b_{1}$] (b1) at (0.00,0) {};
  \node[hollow,label=below:$b_{2}$] (b2) at (1.00,0) {};
  \node[filled,label=below:$b_{3}$] (b3) at (2.00,0) {};
  \node[filled,label=below:$b_{4}$] (b4) at (3.00,0) {};
  \node[filled,label=below:$b_{5}$] (b5) at (4.00,0) {};
  \node[filled,label=below:$b_{6}$] (b6) at (5.00,0) {};
  \node[filled,label=below:$b_{7}$] (b7) at (6.00,0) {};
  \draw[hedge] (a5) -- (b3);
  \draw[hedge] (a5) -- (b4);
  \draw[hedge] (a5) -- (b5);
  \draw[hedge] (a5) -- (b6);
  \draw[hedge] (a5) -- (b7);
  \draw[pair] (a1) to[bend right=35] (a2);
  \draw[pair] (a3) to[bend right=35] (a4);
  \draw[pair] (b1) to[bend left=35] (b2);
  \draw[blue,dotted,very thick] (a2) to[out=-70,in=-110,looseness=2.4] node[below,font=\footnotesize,text=blue] {$\{a_2,a_3\}\notin D_\mathcal A$} (a3);
  \node[font=\footnotesize,anchor=west] at (4.15,2.6) {$=a^*$};
\end{tikzpicture}
\caption{The structure in Case~(c2) of the proof of Proposition~\ref{prop:general-lower}, for $m=5$ and $n=7$: $D_\mathcal B=\{B_0\}=\{\{b_1,b_2\}\}$, the unique uncovered vertex $a^*=a_5$ is joined in $H$ to every vertex of $B\setminus B_0$, and $|D|=x+y+z=5+2+1=8=M_0-1$. Then $A_0=A\setminus\{a^*\}$ has $\binom{m-1}2=6$ pairs, but only $y=2$ of them lie in $D_\mathcal A$; for instance, the vertex $\{a_2,a_3\}$ of $F_2(K_{5,7})$ (dotted curve) is not dominated, so condition~(i) of Lemma~\ref{lem:char} fails.}
\label{fig:c2}
\end{figure}
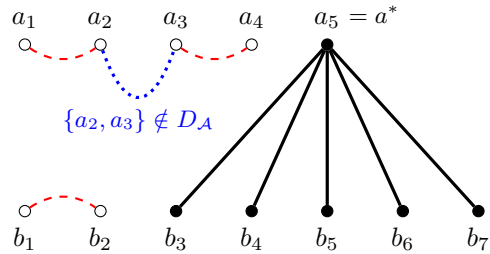

As all the cases (a), (b), (c1), (c2) lead to a contradiction, we have $|D|\ge M_0$.
\end{proof}

The upper bound of Subsection~\ref{sec:odd-general} and Proposition~\ref{prop:general-lower} prove Theorem~\ref{thm:main} for $m\ge5$ odd and $n\ne m+1$. This completes the proof of Theorem~\ref{thm:main}.

\begin{remark}\label{rem:m3}
The argument of Case~(c2) also explains why $m=3$ behaves differently from larger odd values of $m$. In that case, a dominating set $D$ with $|D_\mathcal A|=\frac{m-1}2$ and $|D|=M_0-1$ is forced to have the following structure: $\overline{C}_\mathcal A=\{a^*\}$, $D_\mathcal B=\{B_0\}$ with $|B_0|=2$, and $D_\mathcal X=\{\{a^*,b\}:b\in B\setminus B_0\}$. Such a set is dominating if and only if all the $\binom{m-1}2$ pairs of $A\setminus\{a^*\}$ belong to $D_\mathcal A$, that is, if and only if $\binom{m-1}2\le\frac{m-1}2$, which holds precisely when $m\le3$. For $m=3$ we have $\binom{2}{2}=1=\frac{m-1}2$, and this configuration is exactly the dominating set of size $n=M_0-1$ given in Proposition~\ref{prop:m3}. Hence the value $\gamma(F_2(K_{3,n}))=n$ is one unit below the general formula $\lceil m/2\rceil+n-1$ for every $n\ge3$, whereas for $m\ge5$ this saving is impossible and the only exception is the one at $n=m+1$.
\end{remark}

\section{Computational verification}\label{sec:computation}

As an independent check of Theorem~\ref{thm:main}, we computed $\gamma(F_2(K_{m,n}))$, with the help of Claude (Anthropic), for all $2\le m\le n\le10$\rev{, by solving the standard integer linear program for the domination number: for each vertex $v$ of $F_2(K_{m,n})$ we introduce a binary variable $x_v$, and we minimize $\sum_v x_v$ subject to $\sum_{u\in N[v]}x_u\ge1$ for every vertex $v$, where $N[v]$ denotes the closed neighborhood of $v$. We solved these programs with the mixed-integer solver of SciPy (HiGHS)}. In all $45$ cases, the optimal value coincides with the value given by Theorem~\ref{thm:main}; the results are shown in Table~\ref{tab:values}. 

\begin{table}[ht]
\centering
\begin{tabular}{c|ccccccccc}
\toprule
$m\backslash n$ & $2$ & $3$ & $4$ & $5$ & $6$ & $7$ & $8$ & $9$ & $10$\\
\midrule
$2$ & 2 & 3 & 4 & 5 & 6 & 7 & 8 & 9 & 10 \\
$3$ &  & 3 & 4 & 5 & 6 & 7 & 8 & 9 & 10 \\
$4$ &  &  & 5 & 6 & 7 & 8 & 9 & 10 & 11 \\
$5$ &  &  &  & 7 & \textbf{7} & 9 & 10 & 11 & 12 \\
$6$ &  &  &  &  & 8 & 9 & 10 & 11 & 12 \\
$7$ &  &  &  &  &  & 10 & \textbf{10} & 12 & 13 \\
$8$ &  &  &  &  &  &  & 11 & 12 & 13 \\
$9$ &  &  &  &  &  &  &  & 13 & \textbf{13} \\
$10$ &  &  &  &  &  &  &  &  & 14 \\
\bottomrule
\end{tabular}
\caption{Values of $\gamma(F_2(K_{m,n}))$ for $2\le m\le n\le 10$, obtained by integer linear programming. The entries in bold correspond to the exceptional case $n=m+1$ with $m\ge5$ odd.}
\label{tab:values}
\end{table}

\section{Concluding remarks and further work}\label{sec:conclusion}

We determined $\gamma(F_2(K_{m,n}))$ exactly for all $2\le m\le n$, uncovering a genuine exceptional point at $n=m+1$ when $m\ge5$ is odd. It is natural to ask whether an analogous exceptional behavior occurs for $\gamma(F_2(G))$ when $G$ is a complete multipartite graph with three or more parts; this case is studied in the work in progress~\cite{CR-multi}. Another natural problem is to determine $\gamma(F_k(K_{m,n}))$ for $k\ge3$; in that case the vertices of $F_k(K_{m,n})$ split into $k+1$ classes according to the number of tokens lying in $A$, and the problem acquires a different flavor. This case is considered in the work in progress~\cite{R-Fk}.

%

\section*{Funding}
No funding was received for this work.

\section*{Disclosure statement}
The authors report that there are no competing interests to declare.

\section*{Data availability statement}
The code used for the computational verification in Section~\ref{sec:computation} is available from the corresponding author upon request.

\section*{Declaration of generative AI use}
\rev{During the preparation of this work, the authors used Claude Opus 5.5 (Anthropic) to obtain the auxiliary Lemma~\ref{lem:gmono}, which was used to simplify some of the proofs, and to assist with \LaTeX{} formatting, language editing, suggestions on the organization and presentation of some proofs, and the implementation of the computational verification in Section~\ref{sec:computation}. The authors verified the statement and proof of Lemma~\ref{lem:gmono} and of all other results, and take full responsibility for the content of this article.}

\bigskip
\noindent E. Acosta Troncoso\\
\url{emmanuel_acosta22@yahoo.com}\\
Universidad Aut\'onoma de Zacatecas\\
Unidad Acad\'emica de Matem\'aticas\\
Calzada Solidaridad y Paseo La Bufa, Col.\ Hidr\'aulica\\
C.P.\ 98060, Zacatecas, Zac., Mexico

\medskip
\noindent I. Caldera Caldera\\
\url{isaaccaldera56@gmail.com}\\
Universidad Aut\'onoma de Zacatecas\\
Unidad Acad\'emica de Matem\'aticas\\
Calzada Solidaridad y Paseo La Bufa, Col.\ Hidr\'aulica\\
C.P.\ 98060, Zacatecas, Zac., Mexico

\medskip
\noindent L. M. Rivera (corresponding author)\\
\url{luismanuel.rivera@gmail.com}\\
Universidad Aut\'onoma de Zacatecas\\
Unidad Acad\'emica de Matem\'aticas\\
Calzada Solidaridad y Paseo La Bufa, Col.\ Hidr\'aulica\\
C.P.\ 98060, Zacatecas, Zac., Mexico

\end{document}